%% file: avoidable-vertices-main.tex
\documentclass[11pt,a4paper,abstracton]{scrartcl}

\usepackage{style}

\begin{document}
	
	\title{Avoidable Vertices and Edges in Graphs\thanks{An extended abstract of this work is to appear in the proceedings of
16th International Symposium on Algorithms and Data Structures (WADS 2019).}}
	
	\author[1]{Jesse Beisegel}
	\author[2]{Maria Chudnovsky}
	\author[4]{Vladimir Gurvich}
	\author[5,6]{Martin Milani\v c}
	\author[7]{Mary Servatius}
	\affil[1]{\normalsize BTU Cottbus-Senftenberg, Cottbus, Germany}
	\affil[2]{\normalsize Princeton University, Princeton, NJ, USA}
	\affil[4]{\normalsize National Research University, Higher School of Economics, Moscow, Russia}
	\affil[5]{\normalsize University of Primorska, UP IAM, Muzejski trg 2, SI-6000 Koper, Slovenia}
	\affil[6]{\normalsize University of Primorska, UP FAMNIT, Glagolja\v ska 8, SI-6000 Koper, Slovenia}
	\affil[7]{Koper, Slovenia}
	
	\date{\today}
	
	\maketitle
	
	\begin{abstract}
\begin{sloppypar}
A vertex in a graph is simplicial if its neighborhood forms a clique. We consider three generalizations of the concept of simplicial vertices: avoidable vertices (also known as \textit{OCF}-vertices), simplicial paths, and their common generalization avoidable paths, introduced here. We present a general conjecture on the existence of avoidable paths. If true, the conjecture would imply a result due to Ohtsuki, Cheung, and Fujisawa from 1976 on the existence of avoidable vertices, and a result due to Chv\'atal, Sritharan, and Rusu from 2002 the existence of simplicial paths. In turn, both of these results generalize Dirac's classical result on the existence of simplicial vertices in chordal graphs.

We prove that every graph with an edge has an avoidable edge, which settles the first open case of the conjecture.
We point out a close relationship between avoidable vertices in a graph and its minimal triangulations, and identify new algorithmic uses of avoidable vertices, leading to new polynomially solvable cases of the maximum weight clique problem in classes of graphs simultaneously generalizing chordal graphs and circular-arc graphs.
Finally, we observe that the proved cases of the conjecture have interesting consequences for highly symmetric graphs: in a vertex-transitive graph every induced two-edge path closes to an induced cycle, while in an edge-transitive graph every three-edge path closes to a cycle and every induced three-edge path closes to an induced cycle.
\end{sloppypar}
\end{abstract}
	
	\section{Introduction}
	
	\input{introduction}

	\section{Preliminaries}\label{sec:prelim}

	\input{preliminaries}

	\section{Characterization and Existence of Avoidable Vertices}\label{sec:concepts}
	
	\input{concepts}

	\section{Computing Avoidable Vertices}\label{sec:faster-computation}

	\input{faster-computation}

	\section{Implications for the Maximum Weight Clique Problem}\label{sec:applications}
	
	\input{applications}

    \section{Avoidable Edges in Graphs}

    \label{sec:edges}

	\input{edges}

    \section{Consequences for Highly Symmetric Graphs}

    \label{sec:symmetric}

	\input{symmetric}

	\section{Conclusion}\label{sec:conclusion}
	
	\input{open-problems}

	\subsection*{Acknowledgements}\label{sec:ack}
	
    The authors are grateful to Ekkehard K\"ohler, Matja\v z Krnc, Irena Penev, Robert Scheffler, and Mikhail Vyalyi for interest in their work and helpful remarks, and to Gordon Royle for providing details about reference~\cite{MR1088288}. The work for this paper was done in the framework of two bilateral projects between Germany and Slovenia, financed by DAAD and the Slovenian Research Agency (BI-DE/$17$-$19$-$18$ and BI-DE/$19$-$20$-$04$). The third named author was partially funded by Russian Academic Excellence Project `5-100'. The work of the fourth named author is supported in part by the Slovenian Research Agency (I0-0035, research program P1-0285 and research projects J1-9110, N1-0102). Part of the work was done while the fourth named author was visiting Osaka Prefecture University in Japan, under the operation Mobility of Slovene higher education teachers 2018--2021, co-financed by the Republic of Slovenia and the European Union under the European Social Fund.

	\bibliographystyle{plainnat}
	\bibliography{avoidable-vertices}
	
\end{document}

%% file: introduction.tex

A graph $G$ is \emph{chordal} if every cycle in $G$ of length at least four has a chord. Chordal graphs are well-known to possess many good structural and algorithmic properties~\cite{MR1320296,MR0130190,MR0186421,MR2063679}. The main goal of this paper is to study certain concepts related to chordal graphs in the framework of more general graph classes. The starting point for our research is a result due to Dirac~\cite{MR0130190} stating that
every minimal cutset in a chordal graph is a clique, which
implies that every chordal graph has a \emph{simplicial vertex}, that is, a vertex whose neighborhood is a clique~\cite{MR0186421}. Denoting by $P_k$ the $k$-vertex path, this result can be formulated as follows.

\begin{theorem}\label{thm0}
Every chordal graph has a vertex $v$ such that $v$ is not the midpoint of any induced $P_3$.
\end{theorem}

This theorem was generalized in the literature in various ways, see, e.g.,~\cite{ohtsuki1976minimal,MR956561,MR3115297,MR1927566,MR1971502}.
Two particular ways of generalizing \Cref{thm0} include:
\begin{enumerate}[(i)]
  \item proving a property of general graphs that, when specialized to chordal graphs, results in the existence of a simplicial vertex, and
  \item generalizing the `simpliciality' property from vertices, which are paths of length $0$, to longer induced paths, and proving the existence of such paths for graphs excluding suitably longer cycles.
\end{enumerate}

Let us explain in more detail the corresponding results.

\subsection{First generalization -- from chordal graphs to all graphs}

A generalization of the first kind is given by the following theorem, which follows from~\cite[Theorem 3]{ohtsuki1976minimal} as well as from \cite[Main Theorem~4.1]{berry1998separability} and~\cite[Lemma 2.3]{aboulker2015vertex}.

\begin{theorem}\label{thm1}
Every graph $G$ has a vertex $v$ such that every induced $P_3$ having $v$ as its midpoint is contained in an induced cycle in $G$.
\end{theorem}

The above property of vertices will be one of the central concepts for this paper and we formalize it as follows.

\begin{definition}\label{def:avoidable-vertex}
A vertex $v$ in a graph $G$ is said to be \emph{avoidable} if between any pair $x$ and $y$ of neighbors of $v$ there exists an $x,y$-path, all the internal vertices of which avoid $v$ and all neighbors of $v$. Equivalently, a vertex $v$ is \emph{avoidable} if every induced $P_3$ with midpoint $v$ closes to an induced cycle.
\end{definition}

This terminology is motivated by considering a setting where $G$ represents a symmetric acquaintance relation on a group of people. In this setting, the property of person (equivalently, vertex) $a$ being avoidable can be interpreted as follows: whenever two acquaintances of $a$ need to share some information that they would not like to share with $a$, they can do so by passing the information along a path completely avoiding both $a$ and all her other acquaintances. Thus, $a$ is in a sense avoidable, as information can be passed around in her immediate proximity without her knowledge.

\begin{sloppypar}
Note that every simplicial vertex in a graph is avoidable. If we analyze avoidable vertices in graph classes, rather than in general graphs, we see that this definition is a generalization of many well known concepts. For example, in a tree a vertex is avoidable if and only if it is a leaf, while in a chordal graph a vertex is avoidable if and only if it is simplicial. With this terminology, \Cref{thm1} can be equivalently stated as follows.
\end{sloppypar}

\begin{theorem}\label{thm:avoidable}
Every graph has an avoidable vertex.
\end{theorem}

The notion of avoidable vertices has appeared in the literature (with different terminology) in a variety of settings. To our knowledge, the earliest appearance was in the paper from 1976 by \citet{ohtsuki1976minimal}, where avoidable vertices were characterized as exactly the vertices from which a minimal elimination ordering can start.
Here, a \emph{minimal elimination ordering} of a graph $G = (V,E)$ is a procedure of eliminating vertices one at a time so that before each vertex is removed, its neighborhood is turned into a clique, and the resulting set $F$ of edges added throughout the procedure is an inclusion-minimal set of non-edges of $G$ such that $(V,E\cup F)$ is a chordal graph (in other words, $(V,E\cup F)$ is a \emph{minimal triangulation} of $G$). Given a graph $G$, an avoidable vertex in $G$ can be found in linear time using graph search algorithms such as Lexicographic Breadth First Search (LBFS)~\cite{rose1976algorithmic} (see also~\cite{MR1745069}) or Maximum Cardinality Search (MCS)~\cite{berry2004maximum}. The presentation closest to our setting is the one used by~\citet{ohtsuki1976minimal}. In fact, \citet{berry2004maximum,berry2010graph} named avoidable vertices \emph{OCF-vertices}, after the initials of the three authors of \cite{ohtsuki1976minimal}.

\subsection{Second generalization -- from vertices to longer paths}

In order to generalize the notion of simplicial vertices to longer paths, the next definition, partially following~\citet{MR1927566}, will be useful.

\begin{definition}\label{def:2-extension}
Given an induced path $P$ in a graph $G$, an \emph{extension} of $P$ is any induced path in $G$ obtained by adding to $P$ one edge at each end.
An induced path is said to be \emph{simplicial} if it has no extension.
\end{definition}

\begin{sloppypar}
In this terminology, \Cref{thm0} can be stated as follows: every graph without induced cycles of length more than $3$ has a simplicial induced $P_1$. Chv\'atal et al.~\cite{MR1927566} generalized this result as follows.
\end{sloppypar}

\begin{theorem}[\citet{MR1927566}]\label{thm:higher-chordality}
For each $k\ge 1$, every $\{C_{k+3}, C_{k+4}, \ldots\}$-free graph that contains an induced $P_k$ also contains a simplicial induced $P_k$.
\end{theorem}

\subsection{A common generalization?}

Theorems~\ref{thm:avoidable} and~\ref{thm:higher-chordality} suggest that a further common generalization might be possible, based on the following generalization of~\Cref{def:avoidable-vertex} (definition of avoidable vertices) to longer paths.

\begin{definition}\label{def:avoidable-paths}
An induced path $P$ in a graph $G$ is said to be \emph{avoidable} if every extension of $P$ is contained in an induced cycle.
\end{definition}

Thus, in particular, a vertex $v$ in a graph $G$ is avoidable if and only if the corresponding one-vertex path is avoidable. Moreover, every simplicial induced path is (vacuously) avoidable.

We conjecture that the following common generalization of
Theorems~\ref{thm:avoidable} and~\ref{thm:higher-chordality} holds.

\begin{conjecture}\label{conj:paths}
For every $k\ge 1$, every graph that contains an induced $P_k$ also contains an avoidable induced $P_k$.
\end{conjecture}

\Cref{thm:avoidable} implies the conjecture for $k = 1$, while \Cref{thm:higher-chordality} implies it for every positive integer $k$, provided we restrict ourselves to the class of graphs without induced cycles of length more than $k+2$. Indeed, if $G$ is a $\{C_{k+3}, C_{k+4}, \ldots\}$-free graph that contains an induced $P_k$, then by \Cref{thm:higher-chordality} graph $G$ contains a simplicial induced $P_k$, and every simplicial induced path is avoidable.

\subsection{Our results}

The results of this paper can be summarized as follows.

\medskip
\begin{sloppypar}
\noindent\textbf{1. Characterization, existence, and computation of avoidable vertices.} Following the work of~\citet{ohtsuki1976minimal}, we revisit the connection between avoidable vertices and minimal triangulations of graphs by characterizing avoidable vertices in a graph $G$ as exactly the simplicial vertices in some minimal triangulation of $G$ (\Cref{thm:characterization-avoidable}). Using properties of LBFS that follow from works of~\citet{berry1998separability} and \citet{aboulker2015vertex}, we show that every graph with at least two vertices contains a diametral pair of avoidable vertices, i.e., a pair of avoidable vertices with largest distance from one another (\Cref{thm:two-avoidable}). The same approach shows that a pair of distinct avoidable vertices (though not necessarily a diametral pair) in a given graph $G$ with at least two vertices can be computed in linear time (\Cref{thm:computation}).
\end{sloppypar}

\medskip
\begin{sloppypar}
\noindent\textbf{2. New polynomially solvable cases of the maximum weight clique problem.} A graph is $1$-perfectly orientable if its edges can be oriented so that the out-neighborhood of every vertex induces a tournament, and hole-cyclically orientable if its edges can be oriented so that each induced cycle of length at least four is oriented cyclically. We use the concept of avoidable vertices to show that on hole-cyclically orientable graphs, LBFS computes a bisimplicial elimination ordering, which is simultaneously a circular-arc elimination ordering. This leads to an algorithm that computes a maximum weight clique of a given vertex-weighted hole-cyclically orientable graph $G$ in time ${\mathcal{O}}(|V(G)|(|V(G)|\log |V(G)|+|E(G)|\log\log |V(G)|))$ (\Cref{thm:max-weight-clique}). This result generalizes the well-known fact of polynomial-time solvability of the maximum weight clique problem for the classes of chordal graphs and circular-arc graphs, and gives a polynomial-time algorithm for the maximum clique problem in the class of $1$-perfectly orientable graphs (for which the complexity of the independent set and $k$-coloring problems are still open).
\end{sloppypar}

\medskip
\begin{sloppypar}
\noindent\textbf{3. Existence of avoidable edges.} We show that for every graph $G$ and every non-universal vertex $v\in V(G)$ there exists an avoidable vertex in the non-neighborhood of $v$ (\Cref{thm:avoidable-vertices}). While this result clearly follows from \Cref{thm:two-avoidable}, we give a direct proof that is not based on any graph search (like LBFS). We then adapt this approach to prove the existence of two avoidable edges in any graph with at least two edges (\Cref{thm:avoidable-edges}). This settles in the affirmative the case $k = 2$ of \Cref{conj:paths} and generalizes the case $k = 2$ of \Cref{thm:higher-chordality}.
\end{sloppypar}

\medskip
\begin{sloppypar}
\noindent\textbf{4. Implications for vertex- and edge-transitive graphs.}
We derive some consequences of the above existence results for avoidable vertices and edges for highly symmetric graphs.
More specifically, we show that in a vertex-transitive graph every induced two-edge path closes to an induced cycle (\Cref{thm:VT}), while in an edge-transitive graph every $3$-edge path closes to a cycle (\Cref{thm:ET}) and every induced $3$-edge path closes to an induced cycle (\Cref{cor:edge-transitive-induced-P_4}). Although these structural results are straightforward consequences of the results on avoidable vertices and edges, we are not aware of any statement of these results in the literature. For all the three statements we give examples showing that analogous statements fail for longer paths.
\end{sloppypar}

\subsection{Structure of the paper}

In \Cref{sec:prelim}, we summarize the definitions of some of the most frequently used notions in the paper. In \Cref{sec:concepts}, we discuss structural aspects of avoidable vertices in graphs, including a characterization of avoidable vertices as simplicial vertices in some minimal triangulation of the graph and a new proof of the existence result. \Cref{sec:faster-computation} is devoted to algorithmic issues regarding the problem of efficient computation of avoidable vertices in a given graph. In \Cref{sec:applications}, we present an algorithmic application of this concept to the maximum weight clique problem, by identifying a rather general class of graphs in which every avoidable vertex is bisimplicial, which leads to a polynomial-time algorithm for the maximum weight clique problem in this class of graphs. Implications of this approach for digraphs are also discussed. In \Cref{sec:edges}, we settle in the affirmative the case $k = 2$ of \Cref{thm:avoidable}. Some consequences of the above existence results for avoidable vertices and edges of highly symmetric graphs are summarized in \Cref{sec:symmetric}. We conclude the paper in \Cref{sec:conclusion} with some open problems.

%% file: preliminaries.tex

All graphs considered in this paper will be finite and with non-empty vertex set, but may be either undirected or directed. We will refer to an undirected graph simply as a \emph{graph} and denote it  as $ G=(V,E) $ where $ V $ is the vertex set and $ E $ the edge set. For a graph $G = (V,E)$, we also write $V(G)$ for $V$ and $E(G)$ for $E$. A directed graph will be called a \emph{digraph} and denoted as $ D=(V,A) $ where $ V $ is again the set of vertices and $ A $ the set of arcs. Graphs and digraphs in this paper will always be simple, that is, without loops or multiple edges (but pairs of oppositely oriented arcs in digraphs are allowed). Unless stated otherwise we use standard graph and digraph terminology and notation. In particular, an edge in a graph connecting two vertices $ u $ and $ v $ will be denoted as $ \{u,v\} $ or $ uv $. An arc in a digraph pointing from $ u $ to $ v $ will be denoted as $ (u,v) $ or $ uv $. The set of all vertices adjacent to a vertex $ v $ in $ G $, i.e., its \emph{neighborhood}, is denoted by $ N_G(v) $ (or simply $N(v)$ if the graph is clear from the context). The cardinality of $N_G(v)$ is the \emph{degree} of $ v $, denoted by $ d_G(v) $. Similarly, the \emph{closed neighborhood} $ N_G(v)\cup \{v\} $ is written as $ N_G[v] $ (or simply $N[v]$ if the graph is clear from the context).

Given a digraph $ D=(V,A) $, the \emph{in-neighborhood} of a vertex $ v $ in $ D $, denoted by $ N^-_D(v) $, is the set of all vertices $ w $ such that $ (w,v) \in A $. Similarly, the \emph{out-neighborhood} of $ v $ in $ D $, denoted by $ N^+_D(v) $,  is the set of all vertices $ w $ such that $ (v,w) \in A $. An \emph{orientation} of a graph $ G=(V,E) $ is a digraph obtained by assigning to each edge of $ G $ a direction. A \emph{tournament} is an orientation of the complete graph.

A \emph{clique} (resp., independent set) in a graph $ G $ is a set of pairwise adjacent (resp., non-adjacent) vertices of $ G $.
A \emph{cutset} in a graph $G=(V,E)$ is a set $S$ such that $G - S$ is disconnected. A cutset is \emph{minimal} if it does not contain any other cutset. The \emph{complement} of a graph $ G $ is the graph $ \overline{G} $ with the same vertex set as $ G $, in which two distinct vertices are adjacent if and only if they are not adjacent in $ G $. The \emph{line graph} $ L(G) $ of $ G $ is the graph which has vertex set $ E(G) $ and two vertices $ e $ and $ f $ of $ L(G) $ are adjacent if they have a vertex in common in $ G $. By $P_n$, $C_n$, and $K_n$ we denote the path, cycle, and the complete graph with $n$ vertices, respectively. By $2K_2$ we denote the graph consisting of two disjoint copies of $K_2$. A graph is \emph{bipartite} if its vertex set can be partitioned into two independent sets, which are then said to form a \emph{bipartition} of the graph. By $K_{m,n}$ we denote the complete bipartite graph with $m$ vertices in one part of the bipartition and $n$ in the other one.

Given a graph $ G $ and a set $ S $ of its vertices, we denote by $ G[S] $ the \emph{subgraph of $ G $ induced by $ S $}, that is, the graph with vertex set $ S $ and edge set $ \{uv \in E(G)\mid u,v \in S\} $. By $ G-S $ we denote the subgraph of $ G $ induced by $ V(G)\setminus S $, and when $ S = \{v\} $ we also write $ G-v $. Given an edge set $F\subseteq E(G)$, we  denote by $G-F$ the graph $(V(G),E(G)\setminus F)$; when $F = \{f\}$, we also write $G-f$. Given a family of graphs $\mathcal{F}$, we say that a graph is \emph{$\mathcal{F}$-free} if no induced subgraph of $G$ is isomorphic to a graph in $\mathcal{F}$.

A \emph{walk} of length $ k $ in a graph $ G $ is a sequence of vertices $ (v_1, \ldots , v_{k+1}) $ such that $ v_iv_{i+1} \in E $ for all $ i \in \{1, \ldots , k\} $. If a walk has the additional property that all vertices $ v_i $ for $ i \in \{1, \ldots k\} $ are pairwise distinct, we call it a \emph{path} (in $G$). We will sometimes denote a path $ (v_1, \ldots , v_k) $ as $ v_1- \ldots - v_k $. We say that a path $ P $ \emph{avoids} a vertex $ v $ if the intersection of $ N[v] $ and $ V(P) $ is empty, and we say that $ v $ \emph{intercepts} $ P $ otherwise. A walk $ (v_1, \ldots ,v_k) $ in $G$ such that all vertices $ v_i $ for $ i \in \{1, \ldots k\} $ are pairwise distinct, except that $v_1 = v_k$, is called a \emph{cycle} (in $G$). A \emph{chord} of a path (cycle) in a graph $G$ is an edge $e$ of $G$ between two vertices on the path (cycle) such that $e$ itself does not belong to this path (cycle). A path or a cycle in $G$ is said to be \emph{chordless} if it does not have any chords. We say that a path (resp., induced path) $ (v_1, \ldots , v_k) $ in a graph $G$ \emph{closes} to a cycle (resp., induced cycle) if there is a cycle (resp., induced cycle) in $G$ of the form $(v_1, \ldots , v_k, u_1, \ldots , u_p)$. Given a path $P$ in a graph $G$ and vertices $x,y\in V(P)$, we denote by $x-P-y$ the subpath of $P$ from $x$ to $y$. Concatenations of such paths into longer paths or cycles will be denoted similarly, by $x-P-y-Q-z$, etc.

The \emph{distance} between two vertices $ s $ and $ t $ is the length of a shortest path between these two vertices and will be denoted by $ \dist_G(s,t) $. A vertex $ x $ with largest distance from $ s $ is called \emph{eccentric} to $ s $ and its distance to $ s $ is the \emph{eccentricity} $ \ecc_G(s) $ of $ s $.  The \emph{diameter} of $G$, denoted $\diam(G)$, is the largest such value among all vertices.

\begin{sloppypar}
A permutation $ \sigma = (v_1, \ldots , v_n) $ of the vertices of $ G $ will be called a \emph{vertex ordering}. For a vertex ordering $\sigma = (v_1, \ldots , v_n)$ we write $v_i\prec_{\sigma} v_j$ if $i<j$. Following~\citet{heggernes2006minimal}, a graph $ H= (V, E\cup F) $ is called a \emph{triangulation} of $ G = (V,E)$ if $ H $ is chordal and we say that it is a \emph{minimal triangulation} if for every proper subset $ F' $ of $ F $ the graph $ (V,E \cup F') $ is not chordal. An elimination ordering $ \sigma $ of the vertices of $ G $ is a vertex ordering given as input of the Elimination Procedure, as defined in Algorithm \ref{elimination-procedure}, to compute greedily a triangulation of $ G $ called $ G^+_{\sigma} $. If $ G^+_{\sigma} $ is a minimal triangulation we call $ \sigma $ a \emph{minimal elimination ordering}. If $ G^+_{\sigma} $ is equal to $ G $, then $ \sigma $ is a \emph{perfect elimination ordering} and $ G $ is chordal by definition. The \emph{deficiency} of a vertex $ v $ is defined as the set $ D_G(v)= \{uw \notin E \mid  u,w \in N_G(v)\} $.
\end{sloppypar}

\begin{algorithm}
	\KwIn{A graph $G=(V,E)$ and an ordering $ \sigma = (v_1, \ldots ,v_n) $.}
	\KwOut{The filled graph $ G^+_{\sigma} $.}
	\vspace{0.25cm}
		$ G^0 =G $\;
		
		\For{$ i=1 $ \KwTo $ n $}{
			Let $ F^i = D_{G^{i-1}}(v)$\;
			Obtain $ G^i $ by adding the edges in $ F^i $ to $ G^{i-1} $ and removing $ v_i $\;
		}
		
		$ G^+_{\sigma} = (V, E \cup \bigcup_{i=1}^{n} F^i) $
		\vspace{0.25cm}
	\caption{Elimination Procedure}
	\label{elimination-procedure}
\end{algorithm}

A graph is \emph{cobipartite} if its complement is bipartite and \emph{circular arc} if it is the intersection graph of arcs on a circle. For further terms related to graphs and graph classes, we refer the reader to~\cite{MR1367739,MR1971502,MR2063679,MR1686154}.

%% file: concepts.tex
The proof of~Theorem 3 in the paper~\cite{ohtsuki1976minimal} by
Ohtsuki, Cheung, and Fujisawa (which itself relied on earlier works of Rose~\cite{MR0270957,MR2939965,MR0341833}) leads to the characterization of avoidable vertices given by the following theorem. Since we are not aware of any explicit statement of this result in the literature, we state it here and give a short self-contained proof that does not rely on the concept of minimal elimination orderings.

\begin{theorem}\label{thm:characterization-avoidable}
Let $G = (V,E)$ be a graph and let $v\in V$. Then $v$ is avoidable in $G$ if and only if $v$ is a simplicial vertex in some minimal triangulation of $G$.
\end{theorem}

\begin{proof}
First we show that the condition is sufficient. Let $G' = (V,E\cup F)$ be a minimal triangulation of $G$ and let $v\in V$ be a simplicial vertex in $G'$. Suppose for a contradiction that $v$ is not avoidable in $G$. Then, $v$ contains two neighbors, say $x$ and $y$, such that $x$ and $y$ belong to different connected components of the graph $G- S$, where $S= N_G[v]\setminus \{x,y\}$. Since $v$ is simplicial in $G'$, set $S$ is a clique in $G'$. Let $F^*$ be the set of all pairs $\{u,w\}\in F$ such that $u$ and $w$ are in different connected components of the graph $G - S$ and let $G^*$ be the graph $(V,E\cup (F\setminus F^*))$.
As $S$ is a clique in $G^*$, no induced cycle of $G^*$ contains vertices from two different components of $G^*- S$. It follows that every induced cycle in $G^*$ is also an induced cycle in $G'$, and the fact that $G'$ is chordal implies that $G^*$ is chordal. However, since the set $F^*$ is not empty (note that it contains $\{x,y\}$), the fact that $G^*$ is chordal contradicts the assumption that $G' = (V,E\cup F)$ is a minimal triangulation of $G$. This shows that $v$ is avoidable in $G$.

It remains to show that the condition is also necessary. Let $v\in V$ be an avoidable vertex in $G$ and let $G^*$ be the graph obtained from $G$ by adding edges so that the set $V\setminus \{v\}$ becomes a clique. Since $G^*$ is a chordal graph having $G$ as a spanning subgraph, there exists a minimal triangulation $G' = (V,E\cup F)$ of $G$ such that $E\cup F\subseteq E(G^*)$. We claim that $v$ is a simplicial vertex in $G'$.
As $N_{G'}(v) = N_{G^*}(v) = N_G(v)$, it suffices to show that any two distinct vertices in $N_G(v)$ are adjacent in $G'$. Consider two distinct vertices $x,y$ in $N_G(v)$ and suppose for a contradiction that they are not adjacent in $G'$.
Then they are also not adjacent in $G$. Let $S = N_G[v]\setminus \{x,y\}$.
Since $v$ is avoidable in $G$, vertices $x$ and $y$ are in the same component of
the graph $G-S$. As $N_{G'}(v) = N_G(v)$, we have $N_G[v]\setminus \{x,y\} = N_{G'}[v]\setminus \{x,y\}$, which implies that $x$ and $y$ are in the same component of the graph $G'-S$.
But now, any shortest $x,y$-path $P$ in $G'-S$ can be combined with edges $yv$ and $vx$ into an induced cycle of length at least four in the graph $G'$, contradicting the fact that $G'$ is chordal.
\end{proof}

\begin{remark}
Sets of vertices of a graph $G$ that are maximal cliques in some minimal triangulation of $G$ were studied in the literature under the name \emph{potential maximal cliques}. This concept was introduced by
Bouchitt{\'e} and Todinca in~\cite{MR1857397} and has already found
many applications in algorithmic graph theory (see, e.g.,~\cite{MR1896345,MR2853937,MR3311877}). In this terminology, \Cref{thm:characterization-avoidable} states that given a vertex $v\in V(G)$, its closed neighborhood $N_G[v]$ is a potential maximal clique in $G$ if and only if $v$ is avoidable.
\end{remark}

Since every graph has a minimal triangulation, Theorems~\ref{thm0} and \ref{thm:characterization-avoidable} imply \Cref{thm:avoidable}. An application of \Cref{thm:avoidable} to vertex-transitive graphs will be given in \Cref{sec:symmetric}.

We now discuss the consequence of \Cref{thm:avoidable} when the theorem is applied to the line graph of a given graph. An edge $e$ in a graph $ G $ is said to be \emph{pseudo-avoidable} if the corresponding vertex in the line graph $ L(G) $ is avoidable. It is not difficult to see that an edge $e$ in a graph $G$ is pseudo-avoidable if and only if any (not necessarily induced) $3$-edge path having $e$ as the middle edge closes to a (not necessarily induced) cycle in $G$. Note that the concepts of avoidable edges (considered as induced $P_2$s, in the sense of \Cref{def:avoidable-paths}) and of pseudo-avoidable edges are incomparable, see Fig.~\ref{fig2}. Assuming notation from Fig.~\ref{fig2}, one can see that $ e $ is not an avoidable edge in $ G $. However, it is pseudo-avoidable, as $ e $ is an avoidable vertex in $ L(G) $. On the other hand, $ f $ is avoidable (even simplicial) in $ G $. However, it is not pseudo-avoidable.

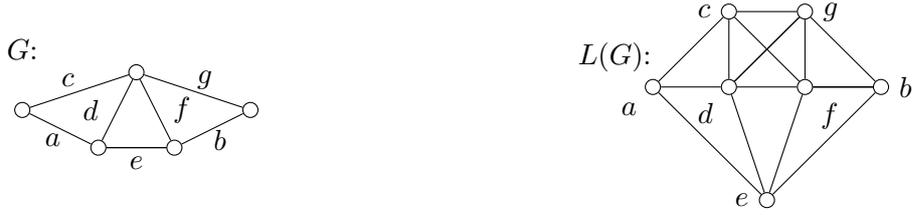
\begin{figure}[ht!]
    \begin{minipage}{0.49\textwidth}
        \centering
        \begin{tikzpicture}[vertex/.style={inner sep=2pt,draw,circle},auto]

        \node[](0) at (0,0.8) {$ G $:};

        \node[vertex](1) at (0,0) {};
        \node[vertex](2) at (1,-0.5) {};
        \node[vertex](3) at (2,-0.5) {};
        \node[vertex](4) at (3,0) {};
        \node[vertex](5) at (1.5,0.5) {};


        \node[](6) at (0.4,-0.4) {$ a $};
        \node[](7) at (2.6,-0.4) {$ b $};
        \node[](8) at (0.6,0.4) {$ c $};
        \node[](9) at (0.9,0) {$ d $};
        \node[](10) at (1.5,-0.7) {$ e $};
        \node[](11) at (2.1,0) {$ f $};
        \node[](12) at (2.4,0.4) {$ g $};

        \draw[] (1)--(2)--(3)--(4)--(5)--(3)--(2)--(5)--(1);

        \end{tikzpicture}
    \end{minipage}
    \hfill
    \begin{minipage}{0.49\textwidth}
        \centering
        \begin{tikzpicture}[vertex/.style={inner sep=2pt,draw,circle},auto]

        \node[](0) at (-0.5,0.4) {$ L(G) $:};

        \node[vertex,label=below left:$ a $](1) at (0,0) {};
        \node[vertex,label=below left:$ d $](2) at (1,0) {};
        \node[vertex,label=below right:$ f $](3) at (2,0) {};
        \node[vertex, label=right:$ b $](4) at (3,0) {};
        \node[vertex, label=right:$ g $](5) at (2,1) {};
        \node[vertex, label=left:$ c $](6) at (1,1) {};
        \node[vertex, label=left:$ e $](7) at (1.5,-1.5) {};

        \draw[] (1)--(2)--(3)--(4)--(7)--(1)--(6)--(2)--(5)--(6)--(3)--(7)--(2)--(5)--(4)--(3)--(5);

        \end{tikzpicture}
    \end{minipage}
    \caption{A graph $ G $ and its line graph $ L(G) $.}\label{fig2}
\end{figure}

Applying \Cref{thm:avoidable} to the line graph of the given graph
yields the following.

\begin{corollary}\label{thm:pseudo-avoidable-edges}
Every graph with an edge has a pseudo-avoidable edge.
\end{corollary}

An application of \Cref{thm:pseudo-avoidable-edges} to edge-transitive graphs will be given in \Cref{sec:symmetric}.

\medskip
Recall that \Cref{thm:avoidable} coincides with the statement of \Cref{conj:paths} for the case $k = 1$. We now reprove this statement in a slightly stronger form, using an approach that we will adapt in~\Cref{sec:edges} for the proof of the case $k = 2$ of the conjecture. A vertex in a graph is said to be \emph{universal} if it is adjacent to every other vertex, and \emph{non-universal} otherwise.

\begin{theorem}\label{thm:avoidable-vertices}
For every graph $G$ and every non-universal vertex $v\in V(G)$ there exists an avoidable vertex	$a\in V(G)\setminus N[v]$.
\end{theorem}
\begin{proof}
Suppose that the theorem is false and take a counterexample $G$ with the smallest possible number of vertices and, subject to that, with the largest possible number of edges. The minimality of $|V(G)|$ implies that $G$ is connected.	Since $G$ is a counterexample, it has a non-universal vertex $v\in V(G)$ such that no vertex not adjacent to $v$ is avoidable.

Let $b\in V(G)\setminus N[v]$. Since $b$ is not avoidable in $G$, it is the midpoint of an induced $P_3$, say $x-b-y$, that is not contained in any induced cycle. Observe that vertices $x$ and $y$ are in different components of the graph $G- (N[b]\setminus\{x,y\})$, since any induced path $P$ from $x$ to $y$ in $G - (N[b]\setminus\{x,y\})$ would imply the existence of an induced cycle $x-b-y-P-x$ closing the $3$-vertex path $x-b-y$. It follows that the graph $G - (N[b]\setminus\{x,y\})$ is disconnected. In particular, there exists an inclusion-minimal subset $S$ of $N(b)$ such that $G -  (S\cup \{b\})$ is disconnected. Note that $v\not\in N[b]$, and, hence, $v$ is a vertex of $G -  (S\cup \{b\})$.	Let $C$ be the component of $G -  (S\cup \{b\})$ containing $v$. It follows from the minimality of $S$ that every vertex in $S$ has a neighbor in $C$.
	
Suppose first that $b$ is not universal in $G -  C$. The minimality of $|V(G)|$ implies that $G -  C$ is not a counterexample
to the theorem. Therefore, there exists an avoidable vertex $a$ in the graph $G -  C$ that is neither equal nor adjacent to $b$.
Since $a$ is a vertex of $G -  (S\cup \{b\})$ not contained in $C$, it is not adjacent to $v$. We claim that $a$ is also avoidable in $G$, which will contradict the assumption that no vertex not adjacent to $v$ in $G$ is avoidable. Let $P$ be an induced $P_3$ in $G$ with midpoint $a$. Since $a$ belongs to a component of $G -  (S\cup \{b\})$ different from $C$, no vertex of $C$ is adjacent to $a$. Therefore, $P$ is an induced $P_3$ in the graph $G -  C$. Since $a$ is avoidable in $G -  C$, there exists an induced cycle in $G -  C$ containing $P$.	Since $G -  C$ is an induced subgraph of $G$, we conclude that there exists an induced cycle in $G$ containing $P$. Since $P$ was arbitrary, it follows that $a$ is avoidable in $G$; a contradiction.
	
Therefore, we may assume that $b$ is universal in $G -  C$. Let $G'$ be the graph obtained from $G -  C$ by adding a new vertex $d$ and making $d$ adjacent to every vertex in $S\cup \{b\}$. Note that if $C\neq \{v\}$, then $G'$ has strictly fewer vertices than $G$. If, on the other hand, $C = \{v\}$, then $G'$ has the same number of vertices as $G$ but strictly more edges, since $N_G(C)\subseteq S$, while $N_{G'}(d) = S\cup \{b\}$. Denoting by $C'$ any component of $G -  (S\cup \{b\})$ other than $C$, we see that every vertex of $C'$ is not adjacent to $d$ in $G'$. Hence, $d$ is not universal in $G'$ and the choice of $G$ implies that $G'$ has an avoidable vertex $a$ that is not $d$ and not adjacent to $d$. This shows that $a\in V(G)\setminus (C\cup S \cup \{b\})$.
	
We claim that $a$ is avoidable in $G$. Suppose that $x,y\in N_G(v)$ are not adjacent. We need to show that the path $x-a-y$ closes to an induced cycle. We have $x,y\in N_{G'}(a)\setminus\{b\}$. Since $a$ is avoidable in $G'$, there exists an induced path $P$ from $x$ to $y$ in $G'$	such that $a$ has no neighbors in $V(P)\setminus\{x,y\}$.	If $d\not\in V(P)$, then $a-x-P-y-a$ is the required cycle in $G$. Therefore, we may assume that $d\in V(P)$. Observe that $d\neq x,y$. Let $p$ and $q$ be the two neighbors of $d$ in $V(P)$.	Then $p,q\in V(G)$. Since $b$ is universal in $G'$ and $b\neq a$, it follows that $b\not\in V(P)$. Since $p$ and $q$ are adjacent to $d$ and $b\not\in V(P)$, it follows that $p,q\in S$. Moreover, notice that $V(P)\setminus \{p,q\}$ is disjoint from $S\cup \{b\}$. By the minimality of $S$ and, since $C$ is connected, there is an induced path $Q$ in $G$ from $p$ to $q$ such that $V(Q)\setminus \{p,q\}\subseteq C$. However, in this case $a-x-P-p-Q-q-P-y-a$ is the required cycle in $G$. This shows that $a$ is an avoidable vertex in $G$ not adjacent to $v$. This contradicts the assumption on $v$ and completes the proof of the theorem.
\end{proof}

%% file: faster-computation.tex

Knowing that every graph has an avoidable vertex, the next question is how to compute one efficiently. The obvious polynomial-time method would be to decide for each vertex $ v $ of the graph $ G $ whether it is avoidable. For this we have to check for each pair of nonadjacent neighbors $ x $ and $ y $ of $ v $, whether they are in the same connected component of $ (G - N[v]) \cup \{x,y\} $. If we use a breadth first search or depth first search to compute the connected components, this gives a running time of $ \mathcal{O}(|V(G)| {|E(\overline{G})|} (|V(G)|+|E(G)|))$. The same method can be used to compute the set of all avoidable vertices. However, if we are only interested in computing one or two avoidable vertices, we show next that this can be done in linear time.

We have already seen that in chordal graphs the avoidable vertices are exactly the same as the simplicial vertices. Therefore, any graph search algorithm that can compute simplicial vertices in a chordal graph is a good candidate for computing avoidable vertices.

\begin{sloppypar}
In 1976,~\citet{rose1976algorithmic} defined a linear-time algorithm (Lex-P), which computes a perfect elimination ordering if there is one, and is thus a recognition algorithm for chordal graphs. This algorithm, since named \emph{Lexicographic Breadth First Search} (LBFS), exhibits many interesting structural properties and has been used as an ingredient in many other  recognition and optimization algorithms on graphs. Any vertex ordering of $G$ that can be produced by LBFS is called an \emph{LBFS ordering} (of $G$).
\end{sloppypar}

\begin{algorithm}
	\KwIn{A connected $n$-vertex graph $G=(V,E)$ and a distinguished vertex $ s \in V $.}
	\KwOut{A vertex ordering $ \sigma $.}
	\vspace{0.25cm}	
		$ label(s) \leftarrow \{n+1\} $\;
		
		\lForEach{vertex $v \in V\setminus\{s\}$}{assign to $ v $ the empty label}
		
		\For{$ i \leftarrow 1 $ \KwTo $ n $}{pick an unnumbered vertex $ v $ with lexicographically largest label\;
			$ \sigma(i) \leftarrow v$\;
			\lForEach{unnumbered vertex $ w \in N(v) $}{append $ (n-i) $ to $ label(w) $}}
		\vspace{0.25cm}
	\caption{Lexicographic Breadth First Search}
	\label{lbfs}
\end{algorithm}

The pseudocode of Lexicographic Breadth First Search given above is presented for connected graphs. However, the method can be generalized to work for arbitrary graphs by executing the search component after component (in an arbitrary order) and concatenating the resulting vertex orderings.

In this context, we will be mainly interested in the properties of the vertices of a given graph $G$ visited \emph{last} by some execution of LBFS, also called \emph{end vertices}. The essential claim of the following lemma can be found in many papers, for example in~\cite{berry1998separability}.

\begin{lemma}[\citet{aboulker2015vertex}]\label{lemma2}
Let $ G=(V,E) $ be a graph and let $ \sigma = (v_1, \ldots , v_n) $ be an LBFS ordering of $G$. Then for all triples of vertices $ a,b,c \in V $ such that $ a \prec_{\sigma} b \prec_{\sigma} c $ and $ ac \in E $, there exists a path from $ a $ to $ b $ whose internal vertices are disjoint from $ N[v_n]$.
\end{lemma}

\begin{corollary}\label{corollary2}
Let $ G=(V,E) $ be a graph and let $ \sigma = (v_1, \ldots , v_n) $ be an LBFS ordering of $G$. Then $ v_n $ is avoidable in $ G $. In fact, for any $ i \in \{1, \ldots ,n\} $ the vertex $ v_i $ is avoidable in $ G[v_1, \ldots, v_i] $.
\end{corollary}

Note that Lexicographic Breadth First Search is a breadth-first search, that is, when LBFS is executed beginning in a vertex $s$, it orders the vertices of $G$ according to their distance from the starting vertex $s$. In particular, this implies the following strengthening of \Cref{thm:avoidable-vertices}.

\begin{corollary}\label{cor:eccentric}
For every graph $ G =(V,E) $ and every vertex $ v \in V $ there is an avoidable vertex $ a \in V $ that is eccentric to $v$.
\end{corollary}

This corollary generalizes the fact that for every vertex $v$ in a chordal graph $G$, there is a simplicial vertex in $G$ that is eccentric to $v$~\cite{MR763692,MR844046}. Moreover, with \Cref{cor:eccentric} at hand we can strengthen \Cref{thm:avoidable} to the following generalization of Dirac's theorem on chordal graphs (\Cref{thm0}).

\begin{theorem}\label{thm:two-avoidable}
Every graph $ G=(V,E) $ with at least two vertices contains two avoidable vertices whose distance to each other is the diameter of $ G $.
\end{theorem}

\begin{proof}
Let $s\in V$ be a vertex of maximum eccentricity in $G$ and let $\sigma = (s=v_1, \ldots ,v_n = a)$ be the ordering given by an LBFS starting in $ s $. By \Cref{corollary2}, vertex $ a $ is avoidable. On the other hand, if $ \tau = (a = w_1, \ldots , w_n = b) $ is an LBFS of $ G $ starting in $ a $, then $ b $ is avoidable due to Corollary \ref{corollary2}. Moreover, $a\neq b$ and
$\dist_G(a,b) = \ecc_G(a) = \ecc_G(s) =\diam(G)$.
\end{proof}

Since LBFS can be implemented to run in linear time~(see, e.g., \cite{MR2063679}), employing the same approach as in the proof of \Cref{thm:two-avoidable}, except that vertex $s$ is chosen arbitrarily,
we obtain the announced consequence for the computation of avoidable vertices.

\begin{theorem}\label{thm:computation}
Given a graph $G$ with at least two vertices, two distinct avoidable vertices in $G$ can be computed in linear time.
\end{theorem}

It is important to note that while every LBFS end vertex is avoidable, the converse is not necessarily true. In the graph depicted in Figure \ref{fig1}(a), vertex $ a $ is avoidable, as it is not contained in any $ P_3 $. On the other hand, $ a $ cannot be the end vertex of an LBFS, as it is not of farthest distance from any of the other vertices.

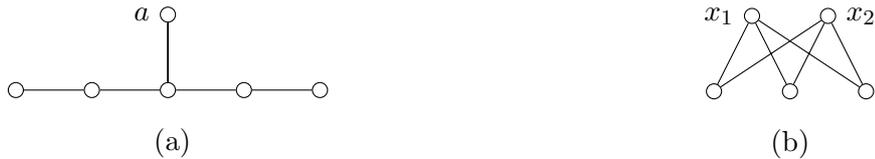
\begin{figure}[ht!]
	\begin{minipage}{0.49\textwidth}
		\centering
		\begin{tikzpicture}[vertex/.style={inner sep=2pt,draw,circle},auto]
		
		\node[vertex](1) at (0,0) {};
		\node[vertex](2) at (1,0) {};	
		\node[vertex](3) at (2,0) {};
		\node[vertex, label=left:$a$](4) at (2,1) {};
		\node[vertex](5) at (3,0) {};
		\node[vertex](6) at (4,0) {};

		\draw[] (1)--(2)--(3)--(4)--(3)--(5)--(6);
		
		\end{tikzpicture}
\medskip

(a)
	\end{minipage}
	\hfill
	\begin{minipage}{0.49\textwidth}
		\centering
		\begin{tikzpicture}[vertex/.style={inner sep=2pt,draw,circle},auto]
		
		\node[vertex, label=left:$x_1$](1) at (0.5,1) {};
		\node[vertex, label=right:$x_2$](2) at (1.5,1) {};	
		\node[vertex](3) at (0,0) {};
		\node[vertex,](4) at (1,0) {};
		\node[vertex](5) at (2,0) {};

		\draw[] (1)--(3)--(2)--(4)--(1)--(5)--(2);
		
		\end{tikzpicture}

\medskip
(b)	
\end{minipage}
	
	\caption{Two graphs having avoidable vertices that cannot be the end vertices of an LBFS, resp.~an MCS.}\label{fig1}
\end{figure}

We can therefore not use LBFS to find all avoidable vertices of a given graph. In fact, while it is possible to check whether a vertex is avoidable in polynomial time (as explained above), it is NP-complete to decide whether an arbitrary vertex of a given graph can be the end vertex of an LBFS \cite{corneil2010end}. This holds true even for restricted graph classes such as bipartite graphs \cite{gorzny2017end} and weakly chordal graphs \cite{corneil2010end}.

\citet{berry2004maximum} give another possibility to compute avoidable vertices in linear time using \emph{maximum cardinality search} (MCS) as defined in Algorithm \ref{mcs}. They show that, analogously to LBFS, the last vertex visited by an MCS is always avoidable (or in their notation an OCF-vertex).

\begin{algorithm}[H]
	\KwIn{Connected graph $G=(V,E)$ and a distinguished vertex $ s \in V $}
	\KwOut{A vertex ordering $ \sigma $}
	\vspace{0.25cm}
			$ \sigma(1) \leftarrow s $\;			
			\For{$ i \leftarrow 2 $ \KwTo $ n $}{pick an unnumbered vertex $ v $ with largest amount of numbered neighbors\;
			$ \sigma(i) \leftarrow v$\;}
		\vspace{0.25cm}
	\caption{Maximum Cardinality Search}
	\label{mcs}
\end{algorithm}

It is easy to see that in the graph given in Figure~\ref{fig1}(a) the set of avoidable vertices is equal to the set of MCS end vertices. However, this is not always the case. In the graph depicted in Figure \ref{fig1}(b), every vertex is avoidable, but neither $ x_1 $ nor $ x_2 $ can be the end vertex of an MCS. Furthermore, just as in the case of LBFS, deciding whether a vertex is an end vertex of MCS is NP-complete~\cite{dmtcs:5572}.

Further analysis on graph searches and avoidable vertices can be found in \cite{berry2010graph}.

%% file: applications.tex

In this section, we present an application of the concept of avoidable vertices to the maximum weight clique problem: given a graph $G = (V,E)$ with a vertex weight function $w:V\to \mathbb{R}_+$, find a clique in $G$ of maximum total weight, where the weight of a set $S\subseteq V$ is defined as $w(S):= \sum_{x\in S}w(x)$. We will show that this problem, which is generally NP-hard, is solvable in polynomial time in the class of $1$-perfectly orientable graphs, and even more generally in the class of hole-cyclically orientable graphs. The importance of these two graph classes, the definitions of which will be given shortly, is due to the fact that they form a common generalization of two well studied graph classes, the chordal graphs and the circular-arc graphs.

The link between avoidable vertices and the classes of $1$-perfectly orientable or hole-cyclically orientable graphs will be given by considering particular orientations of the input graph. Many important graph classes, like chordal graphs, comparability graphs, or proper circular-arc graphs (see, e.g.,~\cite{MR666799}), can be defined with the existence of a particular kind of orientation on the edges. One such class is the class of \emph{cyclically orientable graphs}, first introduced by Barot et al.~\cite{barot2006cluster}. This is the class of graphs that admit an orientation such that every chordless cycle is oriented cyclically.
If we allow triangles to be oriented arbitrarily, while all other chordless cycles must be oriented cyclically, we obtain the class of hole-cyclically orientable graphs. More formally, we say that a \emph{hole} in a graph is a chordless cycle of length at least four, that an orientation $D$ of a graph $G$ is \emph{hole-cyclic} if all holes of $G$ are oriented cyclically in $D$, and that a graph is \emph{hole-cyclically orientable} if it admits a hole-cyclic orientation.

\begin{sloppypar}
While the class of hole-cyclically orientable graphs does not seem to have been studied in the literature, it generalizes the previously studied class of $1$-perfectly orientable graphs, defined as follows. We say that an orientation of a graph is an \emph{out-tournament}, or \emph{$1$-perfect}~\cite{kammer2014approximation}, if the out-neighborhood of every vertex induces a tournament. Similarly, we call a digraph an \emph{in-tournament}~\cite{bangjensen1993tournament}, or \emph{fraternal}~\cite{urrutia1992algorithm}, if the in-neighborhood of every vertex induces a tournament. A graph is said to be \emph{$1$-perfectly orientable} if it admits a $1$-perfect orientation. Using a simple arc reversal argument, it is easy to see that the existence of a $1$-perfect orientation implies a fraternal orientation and vice versa.

The class of $1$-perfectly orientable graphs forms a common generalization of the classes of chordal graphs and of circular-arc graphs~\cite{MR666799,urrutia1992algorithm}. While $1$-perfectly orientable graphs can be recognized in polynomial time via a reduction to a 2-SAT~\cite{bangjensen1993tournament}, their structure is not understood (except for some special cases, see~\cite{bangjensen1993tournament,MR3647815,
MR3853110,MR3634143}) and the complexity of many classical optimization problems such as maximum clique, maximum independent set, or $k$-coloring for fixed $k\ge 3$ is still open for this class of graphs.

In this section, we show that the maximum weight clique problem is solvable in polynomial time in the class of $1$-perfectly orientable graphs.
Moreover, we do this even in the more general setting
of hole-cyclically orientable graphs. The fact that every $1$-perfectly orientable graph is hole-cyclically orientable is a consequence of the following simple lemma (see, e.g., \cite{MR3647815}).
\end{sloppypar}

\begin{lemma}\label{lem:1-p.o.}
Every $1$-perfect orientation of a graph $G$ is hole-cyclic.
\end{lemma}

On the other hand, not every hole-cyclically  orientable graph is $1$-perfectly orientable, as can be seen from Figure \ref{fig4}.

\begin{figure}[ht!]
	\centering
	\begin{tikzpicture}[vertex/.style={inner sep=2pt,draw,circle},auto]
	
	\node[vertex](1) at (0,0) {};
	\node[vertex](2) at (0,1) {};	
	\node[vertex](3) at (1,0) {};
	\node[vertex](4) at (1,1) {};
	\node[vertex](5) at (2,0) {};
	\node[vertex](6) at (2,1) {};

	\draw[->](3)--(5);
	\draw[->](5)--(6);
	\draw[->](6)--(4);
	\draw[->](2)--(4);
	\draw[->](1)--(2);
	\draw[->](3)--(1);
	\draw[->](4)--(3);
		
	\end{tikzpicture}
	\caption{A graph with a hole-cyclic orientation that is not $1$-perfectly orientable.}
	\label{fig4}
\end{figure}
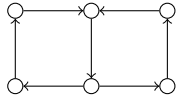

\begin{sloppypar}
Our algorithm for the maximum weight clique problem in the class of hole-cyclically orientable graphs will be based on the fact that the classes of $1$-perfectly orientable and hole-cyclically orientable graphs coincide within the class of cobipartite graphs, where they also coincide with circular-arc graphs. The equivalence between properties 1, 3~and 4 in the lemma below was already observed in~\cite{MR3647815}. Due to \Cref{lem:1-p.o.}, the list can be trivially extended with the hole-cyclically orientable property.
\end{sloppypar}

\begin{lemma}\label{lemma5}
For every cobipartite graph $ G $, the following properties are equivalent:
	\begin{enumerate}
		\item $ G $ is $1$-perfectly orientable.
		\item $ G $ is hole-cyclically orientable.
		\item $ G $ has an orientation in which every induced $4$-cycle is oriented cyclically.
		\item $ G $ is a circular-arc graph.
	\end{enumerate}
\end{lemma}


Another important notion for the algorithm is that of bisimplicial elimination orderings. A vertex $v$ in a graph $G$ is \emph{bisimplicial} if its neighborhood is the union of two cliques in $G$ (or, equivalently, if the graph $G[N(v)]$ is cobipartite). Let $ G=(V,E) $ be a graph and let $ \sigma=(v_1, \ldots,v_n) $ be a vertex ordering of $ G$. We say that $ \sigma $ is a \emph{bisimplicial elimination ordering} of $ G $ if $ v_i $ is bisimplicial in the graph $ G[\{v_1, \ldots, v_i\}] $ for every $ i \in \{1, \ldots, n\} $.

\begin{theorem}[\citet{MR2916349}]\label{prop:bisimplicial}
The maximum weight clique problem is solvable in time $\mathcal{O}(|V(G)|^4)$ in the class of graphs having a bisimplicial elimination ordering.
\end{theorem}

The algorithm can be summarized as follows.

\begin{algorithm}\label{alg3}
	\KwIn{A graph $G=(V,E)$, a weight function $ w:V\to \mathbb{R}_+$, and a bisimplicial elimination ordering $ \sigma=(v_1, \ldots , v_n) $}
	\KwOut{A maximum weight clique $ C^* $ of $ G $}
	\vspace{0.25cm}
		$ C^*:= \emptyset $\;
		
		\For{$ i=0 $ \KwTo $ n-1 $}{
			$ v := \sigma(n-i) $\;
			
			Compute a maximum weight clique $ C_v $ of $ G[N(v)] $\;
			
			\If{$ w(C_v)> w(C^*) $} {$ C^* := C_v$}
			$ G:= G - v $\;
		}
	\vspace{0.25cm}
	\caption{Solving the maximum weight clique problem in graphs with a bisimplicial elimination ordering}
\end{algorithm}

The degree of the polynomial involved in the running time of the algorithm given in~\cite{MR2916349} was not estimated; it was based on polynomial-time solvability of the maximum weight clique problem in the class of perfect graphs. However, the algorithm can be implemented to run in time $\mathcal{O}(|V(G)|^4)$, as follows. Suppose first that the input graph $G$ is equipped with a bisimplicial elimination ordering $(v_1,\ldots, v_n)$. Letting $G_i = G[\{v_1, \ldots, v_i\}] $ for every $ i \in \{1, \ldots, n\} $, at each step of the algorithm, the maximum weight of a clique in $G_i$ is computed by comparing the (recursively computed) maximum weight of a clique in $G_{i-1}$ with the maximum weight a clique containing the current vertex $v_i$. This latter value is computed by solving the maximum weight clique problem in the graph $G_i[N[v_i]]$, which is a cobipartite graph. The maximum weight clique problem in a cobipartite graph $G$ is equivalent to the maximum weight independent set problem in its complement $\overline{G}$, which is bipartite. The maximum weight independent set problem in a bipartite graph with $n'$ vertices and $m'$ edges can by reduced to solving an instance of the maximum flow problem~\cite{MR712925} and can thus be solved in time $\mathcal{O}(n'(n'+m'))$, see, e.g.,~\cite{MR3210838}. It follows that the maximum weight clique problem is solvable in time $\mathcal{O}(n^3)$ in an $n$-vertex cobipartite graph. Consequently, the maximum weight clique problem is solvable in time $\mathcal{O}(n^4)$ in the class of $n$-vertex graphs equipped with a bisimplicial elimination ordering. If a bisimplicial elimination ordering of the input graph is not known, at each step of the algorithm a bisimplicial vertex $v_i$ is first computed in $G_i$. Testing if a vertex $v$ in $G_i$ is bisimplicial can be done in time $\mathcal{O}(|V(G_i)|^2)$ by testing if $\overline{G_i[N(v)]}$ is bipartite; hence, in time $\mathcal{O}(|V(G_i)|^3)$ a bisimplicial vertex $v_i$ in $G_i$ can be found, and a bisimplicial elimination ordering of $G$ can be computed in time $\mathcal{O}(n^4)$.

\medskip
\begin{sloppypar}
As shown by \citet{MR2462311}, every graph without even induced cycles has a bisimplicial vertex. We show next that this property also holds for hole-cyclically orientable graphs. This fact, instrumental to the polynomial-time solvability of the maximum weight clique problem in this class of graphs, is based on a simple argument involving avoidable vertices.
\end{sloppypar}

\begin{lemma}\label{lem:avoidable-bisimplicial}
Every avoidable vertex in a hole-cyclically orientable graph is bisimplicial.
\end{lemma}

\begin{proof}
Let $ G=(V,E) $ be a hole-cyclically  orientable graph, let $D$ be a hole-cyclic orientation of $G$, and let $ a \in V $ be an avoidable vertex in $G$. Suppose that $v,w \in N_D^+(a) $ and $ v $ and $ w $ are not adjacent in $ G $. As $ a $ is avoidable, the path $ v-a-w $ can be closed to an induced cycle $ C $ in $G$. Since $ D$ is a hole-cyclic orientation of $G$ and $C$ is a hole in $G$, we infer that $ C $ is oriented cyclically in $D$. This is a contradiction to the fact that $ v,w \in N_+(a) $. It follows that $ N_D^+(a) $ is a clique in $ G $. The same argument also holds for $ N_D^-(a) $. Since $N_G(a) = N_D^+(a) \cup N_D^-(a)$, this implies that $ a $ is bisimplicial in $ G $.
\end{proof}

\Cref{lem:avoidable-bisimplicial} and \Cref{corollary2} lead to the following.

\begin{sloppypar}
\begin{theorem}\label{thm:bisimplicial-ordering}
Every hole-cyclically orientable graph has a bisimplicial vertex.
Moreover, a bisimplicial elimination ordering of a hole-cyclically orientable graph can be computed in linear time.
\end{theorem}
\end{sloppypar}

\begin{proof}
By \Cref{lem:avoidable-bisimplicial} and \Cref{corollary2}, every LBFS ordering of a hole-cyclically orientable graph is a bisimplicial elimination ordering. Given any graph $G$, an LBFS ordering of $G$ can be computed in linear time (see, e.g., \cite{MR2063679}).
\end{proof}

Theorems~\ref{prop:bisimplicial} and~\ref{thm:bisimplicial-ordering} imply that the maximum weight clique problem is solvable in time
$\mathcal{O}(|V(G)|^4)$ in the class of hole-cyclically orientable graphs. We can improve the running time further using the structure of cobipartite graphs in this class given by \Cref{lemma5}. A \emph{circular-arc elimination ordering} of a graph $G$ is an ordering $(v_1, \ldots,v_n) $ of the vertices of $G$ such that for every $ i \in \{1, \ldots, n\} $, the subgraph of $G$ induced by $\{v_j\mid 1\le j\le i, v_j\in N_G(v_i)\}$ is a circular-arc graph.
By \Cref{lemma5}, every cobipartite hole-cyclically orientable graph is a circular-arc graph. Thus, any bisimplicial elimination ordering of a hole-cyclically orientable graph $G$ is also a circular-arc elimination ordering of $G$.

\begin{sloppypar}
\begin{theorem}\label{thm:max-weight-clique}
The maximum weight clique problem is solvable in time
\hbox{${\mathcal{O}}(n(n\log n+m\log\log n))$}
in the class of hole-cyclically orientable graphs
with $n$ vertices and $m$ edges.
\end{theorem}
\end{sloppypar}

\begin{proof}
A bisimplicial elimination ordering of a given
hole-cyclically orientable graph can be computed in linear time by \Cref{thm:bisimplicial-ordering}. Thus, we can apply Algorithm \ref{alg3}.

Due to \Cref{lemma5}, any cobipartite hole-cyclically orientable graph is a circular-arc graph. Solving the maximum weight clique problem in a given circular-arc graph $G$ equipped with a given circular-arc model can be done in time {${\mathcal{O}}(n\log n+m\log\log n)$}~\cite{MR998461,bhattacharya1997ano}. A circular-arc model of a given circular-arc graph $G$ can be found in time $\mathcal{O}(n+m) $ \cite{mcconnell2003linear}. Using these algorithms as subroutines in Algorithm \ref{alg3} we obtain the stated running time.
\end{proof}

Since every $1$-perfectly orientable graph is hole-cyclically orientable,
\Cref{thm:max-weight-clique} has the following consequence.

\begin{sloppypar}
\begin{corollary}\label{cor:max-weight-clique-1.p.o}
The maximum weight clique problem is solvable in time
\hbox{${\mathcal{O}}(n(n\log n+m\log\log n))$}
in the class of $1$-perfectly orientable graphs
with $n$ vertices and $m$ edges.
\end{corollary}
\end{sloppypar}

In the unweighted case, the same approach as that used in the proof of \Cref{thm:max-weight-clique} results in ${\mathcal{O}}(n(n\log n+m))$ algorithm for the maximum clique problem in the class of hole-cyclically orientable graphs with $n$ vertices and $m$ edges.

\subsection*{A small digression to digraphs}

We conclude this section by showing that the existence of avoidable vertices of graphs along with the approach used in the proof of \Cref{lem:avoidable-bisimplicial} can be applied not only to orientations of simple graphs but also to digraphs in which pairs of oppositely oriented edges are allowed. This leads to results that can be interpreted in the context of information passing in communication networks.

A simple digraph $D = (V,E)$ is \emph{semi-complete} if for every two distinct vertices $u,v\in V$ we have $(u,v)\in E$ or $(v,u)\in E$, and \emph{ out-semi-complete} (resp., \emph{in-semi-complete}) if for every vertex $v\in V$, the subdigraph of $D$ induced by its out-neighborhood $N^+_D(v)$
(resp., in-neighborhood $N^-_D(v)$) is semi-complete. We say that a set of vertices in a digraph is \emph{semi-complete} if it induces a semi-complete digraph. The \emph{underlying graph} of a digraph $D$ is the graph $G$ with vertex set $V(D)$ and $uv\in E(G)$ if and only if $(u,v)\in E(D)$ or $(v,u)\in E(D)$.

\begin{lemma}\label{lem:vertex-in-digraph}
Let $D$ be an out-semi-complete digraph, let $G$ be its underlying graph, and let $v$ be an avoidable vertex in $G$. Then, the in-neighborhood of $v$ in $D$ is semi-complete.
\end{lemma}

\begin{sloppypar}
\begin{proof}
Suppose for a contradiction that $N^-_D(v)$ is not semi-complete.
Then, there is a pair $x,y\in N^-_D(v)$ of non-adjacent vertices in $G$. Since $v$ is avoidable in $G$, the induced path $x-v-y$ extends to an induced cycle $C$ in $G$. Let $(v = v_1,y = v_2, v_3,\ldots, v_k = x, v_1)$ be a cyclic order of vertices on $C$. Then $k\ge 4$.

We claim that for every $i\in \{2,\ldots, k\}$, we have $(v_{i},v_{i-1})\in E(D)$. We show this by induction on $i$. The base case, $i =  2$, holds since $v_2 = y$ is an out-neighbor of $v_1 = v$. Suppose that $(v_{i},v_{i-1})\in E(D)$ for some $i\ge 2$. If $(v_{i+1},v_{i})\not\in E(D)$, then, since $G$ is the underlying graph of $D$, we must have $(v_{i},v_{i+1})\in E(D)$. But then $v_{i-1}$ and $v_{i+1}$ are two non-adjacent vertices in the out-neighborhood of $v_i$, contradicting the assumption that $D$ is out-semi-complete.
This establishes the inductive step and the claim.

But now, since $(v_{k},v_{1})\in E(D)$ and $(v_{k},v_{k-1})\in E(D)$,
vertices $v_{1}$ and $v_{k-1}$ are two non-adjacent vertices in the out-neighborhood of $v_k$, contradicting the assumption that $D$ is out-semi-complete.
\end{proof}
\end{sloppypar}

\Cref{lem:vertex-in-digraph} and~\Cref{thm:avoidable} imply the following.

\begin{sloppypar}
\begin{theorem}\label{cor:digraph}
Every out-semi-complete digraph $D$ contains a vertex whose in-neighborhood is semi-complete. In addition, if $|V(D)|\ge 2$, then $D$ contains at least two vertices whose in-neighborhoods are semi-complete.
\end{theorem}
\end{sloppypar}

The contrapositive statement is the following.

\begin{corollary}\label{cor:digraphs-in-out}
Every digraph with at least two vertices in which at most one in-neighborhood is semi-complete has an out-neighborhood that is not semi-complete.
\end{corollary}

\begin{sloppypar}
A digraph can be viewed as a communication network, where each vertex is both an information source (passing information to its out-neighbors) and an information recipients (receiving information from its in-neighbors).
We say that two information sources / receivers are \emph{independent} if the corresponding vertices are non-adjacent.
Thus, Theorem~\ref{cor:digraph} and Corollary~\ref{cor:digraphs-in-out} can be interpreted as follows:
\begin{itemize}
  \item If no participant passes information to at least two independent recipients, then at least two participants each get information from sources that are not pairwise independent.
  \item If each participant (but possibly one) gets information from at least two independent sources, then there exists a participant who sends information to at least two independent recipients.
\end{itemize}
\end{sloppypar}

%% file: edges.tex
We will call an edge $e$ in a graph $G$ \emph{avoidable} (resp., \emph{simplicial}) if the path $P_2$ induced by its endpoints is avoidable (resp., simplicial). In particular, if an edge $e$ in a graph $G$ is not the middle edge of any induced $P_4$, then $e$ is simplicial and thus avoidable. A sufficient condition for an edge $e = uv$ in a graph $G$ to be simplicial is that it is \emph{bisimplicial}, i.e., that $N(u)\cup N(v)$ induces a complete bipartite graph. Bisimplicial edges are relevant for \emph{perfect  elimination bipartite} graphs, defined as bipartite graphs whose edges can be eliminated by successively removing both endpoints of a bisimplicial edge, and for their subclass \emph{chordal bipartite graphs}, defined as bipartite graphs without induced cycles of length more than $4$, see~\cite{golumbicGoss1978,MR2063679}. Note that an edge in a bipartite graph is simplicial if and only if it is bisimplicial. Thus, the fact that every chordal bipartite graph is perfect elimination bipartite can be equivalently stated as follows: every chordal bipartite graph with an edge has a simplicial edge.

The case $k = 2$ of \Cref{conj:paths} states that every graph with an edge has an avoidable edge. \Cref{thm:higher-chordality} settles this case of the conjecture for $\{C_{5}, C_{6}, \ldots\}$-free graphs; in fact, it asserts that every $\{C_{5}, C_{6}, \ldots\}$-free graph with an edge has a simplicial edge. Since every chordal bipartite graph is $\{C_{5}, C_{6}, \ldots\}$-free, this generalizes the above result for chordal bipartite graphs. A related result is that of~\citet{MR1363691} stating that a graph is weakly chordal (that is, both the graph and its complement are $\{C_{5}, C_{6}, \ldots\}$-free) if and only if its edges can be eliminated one at a time, where each eliminated edge is simplicial in the subgraph consisting of the remaining edges.

In this section, we prove the case $k = 2$ of \Cref{conj:paths} for all graphs. Given a graph $G$, two edges will be called \emph{independent} in $ G $ if their endpoints form an induced $ 2K_2 $ in $ G $. We first consider the case when the graph contains no two independent edges.

\begin{lemma}\label{lemma3}
Let $ G $ be a graph with at least two edges but with no two independent edges. Then $ G $ contains at least two avoidable edges.
\end{lemma}

\begin{proof}
Suppose that the lemma is false and let $G$ be a counterexample minimizing the number of vertices. We may assume that $G$ is not complete, since otherwise any two edges in $G$ are simplicial and thus avoidable. Let $ S $ be a minimal cutset of $ G $.
	
\textbf{Case 1:} \emph{Graph $ G-S $ consists of isolated vertices.}

As $ S $ is a minimal cutset, every one of these (at least two) isolated vertices must be adjacent to every vertex in $ S $. Let $ c $ be such a vertex and let $ s $ be an arbitrary vertex in $ S $. Suppose there is an induced $ P_4 $, say $ P = x-c-s-y $ with $ cs $ as a middle edge. Clearly, vertex $ x $ must be in $ S $ and thus $ y $ also has to be in $ S $, as $ x $ is adjacent to every vertex outside of $ S $. This is a contradiction as $ c $ is adjacent to every vertex in $ S $. Therefore, $ cs $ is avoidable and as there are at least two isolated vertices in $ G-S $ we can find two such avoidable edges.
	
\textbf{Case 2:} \emph{Some connected component of $ G-S $ contains an edge.}

As there are no independent edges in $ G $ there can be only one such connected component, which we will denote with $ C $. Also, in $G-S$ there must be at least one isolated vertex $ c' $. Note that $c'$ is adjacent to every vertex in $ S $. We analyze two further subcases.

\textbf{Case 2a:} \emph{Component $ C $ has exactly two vertices.}

Let $ c_1 $ and $ c_2 $ be the two vertices of $C$.
Any $ P_4 $ with $ c_1c_2 $ as its middle edge must be of the form $ x-c_1-c_2-y $ with $ x $ and $ y $ in $ S $. Since $ c' $ is adjacent to every vertex in $ S $ but to none in $ C $, we can close this path to an induced cycle using $ c' $. Thus, $ c_1c_2 $ is avoidable. Let $ s_1 $ be a an arbitrary vertex in $ S $. Then either $ c's_1 $ is avoidable or it is the middle vertex of an induced $ P_4 $ that does not close to an induced cycle. Without loss of generality such a path is of the form $P= s_2-c'-s_1-c_1 $ where $ s_2 $ is an element of $ S $, as every isolated vertex of $ G-S $ is adjacent to every vertex of $ S $. Since edges $c_1c_2$ and $c's_2$ are not independent in $G$, we infer that $ s_2 $ is adjacent to $ c_2 $. As $ P $ does not close to an induced cycle, $ c_2 $ must be adjacent to $ s_1 $. We claim that in this case edge $ s_2c' $ is avoidable. Suppose $ s_2c' $ is the middle edge of an induced $ P_4 $. This path must be of the form $P' = c_2-s_2-c'-s_3 $ with $ s_3 \in S\setminus\{s_1\} $. Since $S$ is a minimal cutset in $G$, vertex $s_3$ is adjacent to a vertex in $C$. Thus, $ s_3 $ is adjacent to $ c_1 $ and path $P'$ closes to an induced cycle. As a result $ G $ has two avoidable edges.	
	
\textbf{Case 2b:} \emph{Component $ C $ has more than two vertices.}

Then $C$ has at least two edges, but no two independent edges.
Since $ G $ is a minimal counterexample, $C $ has at least two avoidable edges $ e = vw $ and $ f=xy $ which are not necessarily disjoint.
We claim that $e$ and $f$ are also avoidable in $G$. By symmetry, it suffices to show the claim for $e$. Since $e$ is avoidable in $C$, every $ P_4 $ with $ e $ as a middle edge that is completely contained in $ C $ can be closed to an induced cycle in $C$ and thus in $ G $ as well. Suppose that $ e $ is the middle edge of an induced $ P_4 $, say $ P= t_1-v-w-t_2 $, that is not completely contained in $ C $. First we assume that $ t_1 \in S $ and $ t_2 \notin S$. Then $ c'$ is adjacent to $ t_1 $, which implies that edges $c't_1$ and $wt_2$ are independent in $G$; a contradiction. Thus, both $ t_1 $ and $ t_2 $ are in $ S $, and we can close $ P $ to an induced cycle using the isolated vertex $ c' $, which is adjacent to $ t_1 $ and $ t_2 $ but not to any endpoint of $ e $. It follows that $G$  contains two avoidable edges.
\end{proof}

Clearly, if a graph has a single edge, this edge is avoidable. Thus, \Cref{lemma3} has the following consequence.

\begin{corollary}\label{corollary-avoidable-edge}
Every graph with at least one edge but with no two independent edges contains an avoidable edge.
\end{corollary}

To consider the case not settled by~\Cref{lemma3}, we first introduce some more terminology. Two distinct edges will be called \emph{weakly adjacent} if they are not independent. The set $ N^E_G(e) $ will denote the set of edges of $G$ that are weakly adjacent to $ e $. The members of $ N^E_G(e) $ are the \emph{edge-neighbors} of $e$. An edge $ e\in E(G)$ will be called \emph{universal} in $ G $ if every edge of $ G $ other than $e$ is weakly adjacent to $e$. We say that an edge $ e = uv $ is \emph{adjacent to} a vertex $ w $ in $ G $ (and vice versa) if $ w \notin \{u,v\} $ and $ uw $ or $ vw $ is an element of $ E(G) $. The set of vertices that are adjacent to an edge $ e $ is called the \emph{vertex-neighborhood of} $ e $ and is denoted by $ N_G^V(e) $, while the set of edges adjacent to a vertex $ v $ is called the \emph{edge-neighborhood of} $v$ and is denoted by $ N^E_G(v)$. Note that the edge-neighborhood of any edge $e = uv $ is exactly the set of all edges having at least one endpoint in the vertex-neighborhood of $e$.

The case not considered by \Cref{lemma3} is settled in the next lemma.

\begin{lemma}\label{theorem2}
For every graph $ G $ and every non-universal edge $ e \in E(G) $ there is an edge $ f \in E(G) $ independent of $ e $ which is avoidable.
\end{lemma}
\begin{proof}
Suppose that the lemma is false and take a counterexample $G$ with the smallest possible number of vertices. Since $G$ is a counterexample, it has a non-universal edge $e \in E(G)$ such that no edge independent from $ e $ is avoidable.

We first prove a sequence of claims.

\medskip
\begin{claim}\label{claim0}
Every edge that is independent of $ e $ is adjacent to all vertex-neighbors of $ e $ in $ G $.
\end{claim}

\medskip

\begin{claimproof}
Suppose that the claim is false and let $f\in E(G)$ be an edge that is independent of $ e $ and non-adjacent to at least one vertex-neighbor $p$ of $ e $ in $ G $. Let $ G' $ be the graph resulting from contracting the edge $ e $ in $ G $. If $ e' $ denotes the vertex obtained from $ e $ in $ G' $, then $ pe' \in E(G') $ is independent of $ f $ in $ G' $. Since $ G $  is a minimal counterexample and as $ f $ is independent of $ pe' $ in $ G' $, there must be an edge $ g \in E(G') $ that is independent of $ pe' $ and avoidable in $ G' $. It is easy to see that $ g $ is also an edge in $ G $ that is independent of $ e $ and avoidable in $ G $; a contradiction.
\end{claimproof}
	
\medskip
Let $ f=vw $ be an edge that is independent of $ e $ and which has the fewest edge-neighbors among all such edges. By Claim~\ref{claim0}, $ f $ is adjacent to every neighbor of $ e $ and is not avoidable, i.e., it is the middle edge of an induced $ P_4 $, say $ P=x-v-w-y $, that cannot be closed to an induced cycle. This implies in particular that $ x $ and $ y $ cannot both be adjacent to $ e $.

Suppose that $ x $ is adjacent to $ e $. Then $ y $ is not adjacent to $e$. It follows that edge $ wy$ is independent of $ e $ and thus adjacent to all neighbors of $ e $, including $ x $; a contradiction. Therefore, $ x $ is not adjacent to $ e $ and, by symmetry, neither is $y$.

Note that $N_G^V(e)\subseteq N_G^V(f)\setminus \{x,y\}$. Moreover, the set $N_G^V(e)\cup\{v,w\}$ separates $ e $ from $ x $, that is,
$e$ and $x$ belong to different components of the graph $G - (N_G^V(e)\cup \{v,w\})$.  It follows that we can find a minimal set $ S \subseteq N_G^V(e)$ (possibly $S = \emptyset$) such that the set $S\cup \{v,w\}$ separates $e$ from $x$. Let $ C $ be the connected component of $ G - (S\cup \{v,w\}) $ containing $ e $.
 	
\medskip	
	\begin{claim}\label{claim1}
		Edge $ f $ is universal in $ G - C $.
	\end{claim}

\medskip
	\begin{claimproof}
		Suppose not. Then, due to the minimality assumption made on $ G $, there must be an edge $ h \in E(G - C) $ that is independent of $ f $ and avoidable in $ G - C $. As shown above, every vertex in $ S $ is adjacent to $ v $ or $ w $. Therefore, edge $ h $ must be fully contained in a connected component $ C' $ of $ G - (S\cup \{v,w\}) $. As there are no edges between  $ C' $ and $ C $ in $ G $, any induced $ P_4 $ that has $ h $ as a middle edge is contained in $ G - C $ and can be closed to an induced cycle. Thus, $ h $ is also avoidable in $ G $ as well as being independent of $ e $.
	\end{claimproof}
	
\medskip
Claim~\ref{claim1} has the following consequence.
	
\medskip
\begin{claim}\label{claim2}
Every edge $ g \in E(G - C) $ that does not have vertex-neighbors in $ C $ is universal in $ G - C $ and adjacent to every vertex in $S$.
In particular, this holds for any edge that is completely contained in a connected component of $ G - (S\cup \{v,w\}) $ not equal to $ C $.
\end{claim}

\medskip

\begin{claimproof}
Let $ g $ be an edge in $G - C$ that does not have vertex-neighbors (in $G$) outside of $G - C $. Then $g$ is independent of $e$. Recall that edge $ f $ was chosen to be independent of $ e $ with the smallest number of edge-neighbors. By Claim~\ref{claim1}, $f$ is universal in $ G - C $. Note that each edge-neighbor of $ g $ is either contained in $ G - C $ or contains a vertex from $S \cup \{v,w\}$. This means that each edge-neighbor of $ g $ is either equal to $f$ or is an edge-neighbor of $ f $. The choice of $f$ implies $N_G^E(g) = (N_G^E(f)\setminus \{g\})\cup \{f\}$, showing that $g$ is universal in $G- C$.

Let $ s $ be an arbitrary vertex in $ S $. Due to the choice of $ S $, vertex $ s $ must have a neighbor in $ e $, say $ t $, and is adjacent to $ f $. Thus, edge $st$ is an edge-neighbor of $f$. As this edge is also an edge-neighbor of $g$ and $ g $ cannot be adjacent to $ t $, it follows that $g$ is adjacent to $ s $.
\end{claimproof}

\medskip
The next claim restricts the structure of components of the graph $ G - (S \cup \{v,w\}) $.

\medskip
\begin{claim}\label{claim3}
Exactly one component of $ G - (S \cup \{v,w\}) $ other than $ C $ contains an edge.
\end{claim}

\medskip
\begin{claimproof}
Suppose first that at least two connected components of $ G - (S \cup \{v,w\}) $ other than $ C $ contain edges. Consider two such components and one edge in each. By Claim \ref{claim2}, each of these edges is universal in $ G - C $. This contradicts the fact that they are independent. It follows that at most one component of $ G - (S \cup \{v,w\}) $ other than $ C $ contains an edge.

Suppose for a contradiction that no component of $ G - (S \cup \{v,w\}) $ other than $ C $ contains an edge, that is, all such components are trivial. In particular, the component of $ G - (S \cup \{v,w\}) $ containing vertex $x$ is trivial. Consider the edge $g = xv$. We already know that $g$ is independent of $e$. This implies that $g$ is not avoidable in $G$. Thus, there exists an induced $ P_4 $ in $ G $ having $g$ as the middle edge, say $ Q = t_1 - x - v - t_2 $, that does not close to an induced cycle. Clearly, $ t_1 $ must be in $ S $ and thus adjacent to $ e $. It follows that $ t_2 $ cannot be adjacent to $ e $, since otherwise we could close $ Q $ to an induced cycle. But now, edge $ t_2v $ is independent of $ e $ and hence, by Claim~\ref{claim0}, adjacent to every neighbor of $ e $, including $ t_1 $; a contradiction.
\end{claimproof}

\medskip
By Claim~\ref{claim3}, exactly one component of $ G - (S \cup \{v,w\}) $ other than $ C $ contains an edge. Let $C'$ be this component. We complete the proof by considering three exhaustive cases.

\medskip
\textbf{Case 1:} \emph{Both $ v $ and $ w $ have neighbors outside of $ S \cup V(C') \cup \{v,w\} $.}

It follows from Claim~\ref{claim2} that every edge in $C'$ is universal in $C'$. Therefore, \Cref{corollary-avoidable-edge} implies that $C'$ contains an avoidable edge $ g $. Since $g$ is independent of $e$, it is not avoidable in $G$. This means that there exists an induced $ P_4 $ in $ G $ having $g$ as the middle edge, say $ Q = t_1 - g_1 - g_2 - t_2 $, that does not close to an induced cycle. Since $g$ is avoidable in $C'$, path $Q$ cannot be fully contained in $C'$. Without loss of generality, we may assume that $t_1\not\in V(C')$. It follows that $t_1\in S\cup\{v,w\}$.

Suppose that $t_2 \in V(C')$. If $ t_1 \in S $, then $ t_2 $ or $ g_2 $ must be adjacent to $ t_1 $, due to Claim \ref{claim2}, as $ g_2t_2 $ is an edge contained in $ C' $; this is a contradiction. Therefore, $t_1\in \{v,w\}$.
As both $ v $ and $ w $ have neighbors outside of $ S \cup V(C') \cup \{v,w\}$, similar arguments as those used in the proof of Claim \ref{claim2} can be used to show that either $ t_2 $ or $ g_2 $ must be adjacent to $ t_1$; a contradiction. This shows that $t_2\not\in V(C')$ and therefore $t_2\in S\cup\{v,w\}$.

If both $ t_1 $ and $ t_2 $ are in $ S $, then $ Q $ can be closed to an induced cycle through $ e $, as all vertices in $ S $ are adjacent to $ e $. Therefore, we may assume without loss of generality that $ t_1 \in \{v,w\}$ and $ t_2 \in S $. By the assumption of Case 1, vertex $t_1$ has a neighbor $z$ outside of $S \cup V(C') \cup \{v,w\}$. If $z$ is in $C$, then there is a path from $t_1$ to $t_2$ through $C$, and if $z$ is not in $C$, then $z$ is an isolated vertex of $ G - (S \cup \{v,w\}) $ and
$t_1-z-t_2$ is a path in $G$. In either case, using such a path, $ Q $ can be closed to an induced cycle in $G$; a contradiction.

\medskip
\textbf{Case 2:} \emph{Exactly one of $v$ and $w$ has a neighbor outside of $ S \cup V(C') \cup \{v,w\} $.}

We may assume without loss of generality that $ v $ does not have any neighbors outside of $ S \cup V(C') \cup \{v,w\} $, but $ w $ does. Claim \ref{claim2} implies that every edge in $ G[C' \cup \{v\}] $ is universal. Hence, we infer, using \Cref{corollary-avoidable-edge}, that $ G[C' \cup \{v\}] $ contains an avoidable edge $ g $. Since $ v $ does not have any neighbors outside of $ S \cup V(C') \cup \{v,w\}$, edge $g$ (which might contain $v$ as an endpoint) is independent of $e$. This implies that $g$ is not avoidable in $G$, i.e., there exists an induced $ P_4 $ in $ G $ having $g$ as the middle edge, say $ Q = t_1 - g_1 - g_2 - t_2 $, that does not close to an induced cycle. Since $g$ is avoidable in $ G[C' \cup \{v\}]$, path $Q$ cannot be fully contained in $G[C' \cup \{v\}]$. Without loss of generality, we may assume that $t_1\not\in V(C')\cup \{v\}$. It follows that $t_1\in S\cup\{w\}$.

Suppose that $t_2 \in V(C')\cup\{v\}$. If $t_1 \in S $, then $ t_2 $ or $ g_2 $ must be adjacent to $ t_1 $, due to Claim \ref{claim2}; this is a contradiction. Therefore, $t_1= w$. As $w$ has a neighbor outside of $ S \cup V(C') \cup \{v,w\}$, similar arguments as those used in the proof of Claim \ref{claim2} can be used to show that either $ t_2 $ or $ g_2 $ must be adjacent to $t_1$, leading again to a contradiction. This shows that $t_2\not\in V(C')\cup \{v\}$ and therefore $t_2\in S\cup\{w\}$.

If both $ t_1 $ and $ t_2 $ are in $ S $, then $ Q $ can be closed to an induced cycle through $ e $, as all vertices in $ S $ are adjacent to $ e $. Therefore, we may assume without loss of generality that $ t_1 = w$ and $t_2 \in S $. By the assumption of Case 2, vertex $w$ has a neighbor $z$ outside of $S \cup V(C') \cup \{v,w\}$. If $z$ is in $C$, then there is a path from $t_1$ to $t_2$ through $C$. On the other hand, if $z$ is not in $C$, then $z$ is an isolated vertex of $ G - (S \cup \{v,w\}) $ and $t_1-z-t_2$ is a path in $G$. In either case, using such a path $ Q $ can be closed to an induced cycle in $G$; a contradiction.
					
\medskip
\textbf{Case 3:} \emph{Neither $ v $ nor $ w $ has a neighbor outside of $ S \cup V(C') \cup \{v,w\}$.}

Claim \ref{claim2} implies that every edge in $ G[C' \cup \{v,w\}] $ is universal, hence, by \Cref{corollary-avoidable-edge}, there is an avoidable edge $ g $ in $ G[C' \cup \{v,w\}] $. The assumption of Case 3 implies that edge $g$ is independent of $e$. This implies the existence of an induced $ P_4 $ in $ G $ having $g$ as the middle edge, say $ Q = t_1 - g_1 - g_2 - t_2 $, that does not close to an induced cycle. Since $g$ is avoidable in $ G[C' \cup \{v,w\}]$, path $Q$ cannot be fully contained in $G[C' \cup \{v,w\}]$. Without loss of generality, we may assume that $t_1\not\in V(C')\cup \{v,w\}$. It follows that $t_1\in S$.
Moreover, $t_2\in S$, since otherwise $t_2$ would belong to $V(C')\cup\{v,w\}$ and Claim \ref{claim2} would imply that one of $ t_2 $ or $ g_2 $ is adjacent to $ t_1 $; a contradiction. Since both $ t_1 $ and $ t_2 $ are in $ S $, path $ Q $ can be closed to an induced cycle through $ e $, as all vertices in $ S $ are adjacent to $ e $; a contradiction.
	
This completes the proof of the lemma.
\end{proof}

Lemmas~\ref{lemma3} and~\ref{theorem2} imply the following.

\begin{theorem}\label{thm:avoidable-edges}
Every graph with an edge has an avoidable edge. Every graph with at least two edges has two avoidable edges.
\end{theorem}

\begin{proof}
Clearly, if $G$ has a single edge, that edge is avoidable. So let $G$ be a graph with at least two edges. If $G$ does not have two independent edges, then \Cref{lemma3} applies and the desired conclusion follows. If, on the other hand, $G$ does have a pair of independent edges, say $e$ and $e'$, then $e$ is not universal in $G$ and hence by \Cref{theorem2} there is an edge $ f \in E(G) $ independent of $ e $ which is avoidable. Since $f$ is independent of $e$, it is not universal. Applying \Cref{theorem2} again, we find that $G$ contains an avoidable edge $f'$ that is independent of $ f$. Clearly, the edges $f$ and $f'$ are distinct.
\end{proof}

An application of this theorem to edge-transitive graphs is given in the next section.

%% file: symmetric.tex
In this section we summarize the consequences of existence results for avoidable vertices and edges for vertex- and edge-transitive graphs. An \emph{automorphism} of a graph $G$ is a bijection from the vertex set of $G$ to itself that maps edges to edges and non-edges to non-edges. A graph $G$ is said to be \emph{vertex-transitive} if for any two vertices $u,v\in V(G)$ there exists an automorphism of $G$ mapping $u$ to $v$. Similarly, $G$ is said to be \emph{edge-transitive} if for any two edges $e,f\in E(G)$ there exists an automorphism of $G$ mapping $e$ to $f$.

By \Cref{thm:avoidable}, the midpoint of any induced $P_3$ in a vertex-transitive graph is avoidable. This implies the following consequence for vertex-transitive graphs.

\begin{corollary}\label{thm:VT}
Every induced $ P_3 $ in a vertex-transitive graph closes to an induced cycle.
\end{corollary}

As shown by the example in Figure~\ref{fig:VT}, a statement analogous to that of \Cref{thm:VT} fails for longer paths.

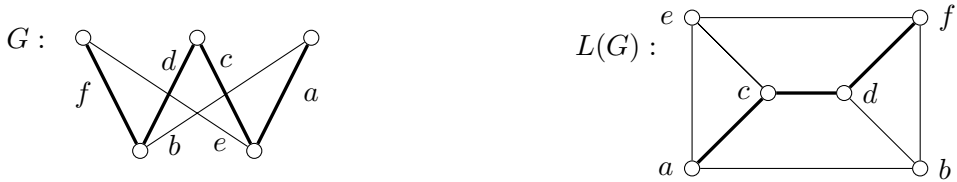
\begin{figure}[ht!]
    \begin{minipage}{0.49\textwidth}
        \centering
        \begin{tikzpicture}[vertex/.style={inner sep=2pt,draw,circle},auto,rotate=90,scale=1.5]

        \node[] (0) at (1,2.5) {$ G: $};

        \node[vertex](1) at (0,0.5) {};
        \node[vertex](2) at (0,1.5) {};
        \node[vertex](3) at (1,0) {};
        \node[vertex](4) at (1,1) {};
        \node[vertex](5) at (1,2) {};

        \node[] (6) at (0.5,0) {$ a $};
        \node[] (7) at (0.8,0.75) {$ c $};
        \node[] (8) at (0.8,1.25) {$ d $};
        \node[] (9) at (0.5,2) {$ f $};
        \node[] (10) at (0.05,0.8) {$ e $};
        \node[] (8) at (0.05,1.2) {$ b $};

        \draw[] (1)--(3)--(2)--(4)--(1)--(5)--(2);
        \draw[line width = 1.25] (5)--(2)--(4)--(1)--(3);
        \end{tikzpicture}
    \end{minipage}
    \begin{minipage}{0.49\textwidth}
        \centering
        \begin{tikzpicture}[vertex/.style={inner sep=2pt,draw,circle},auto]

        \node[] (0) at (-1,1.6) {$ L(G): $};

        \node[vertex,label=left:$ a $](1) at (0,0) {};
        \node[vertex,label=right:$ b $](2) at (3,0) {};
        \node[vertex,label=left:$ c $](3) at (1,1) {};
        \node[vertex,label=right:$ d $](4) at (2,1) {};
        \node[vertex,label=left:$ e $](5) at (0,2) {};
        \node[vertex,label=right:$ f $](6) at (3,2) {};

        \draw[line width =1.25] (1)--(3)--(4)--(6);
        \draw[] (4)--(2)--(6)--(5)--(3)--(5)--(1)--(2);
        \end{tikzpicture}
    \end{minipage}
    \caption{The graph $ L(G) $ is vertex-transitive and contains an induced four-vertex path $(a,c,d,f)$ that does not close to an induced cycle.}
\label{fig:VT}
\end{figure}

A similar consequence can be derived for edge-transitive graphs, by considering avoidable vertices in the line graph of a graph. By \Cref{thm:pseudo-avoidable-edges}, the middle edge of any $3$-edge path in an edge-transitive graph is pseudo-avoidable.

\begin{corollary}\label{thm:ET}
Every $3$-edge path in an edge-transitive graph closes to a cycle.
\end{corollary}

A statement analogous to that of \Cref{thm:ET} fails for longer paths. The graph shown in Figure~\ref{fig:VT} is the line graph of the complete bipartite graph $K_{2,3}$, which is edge-transitive. The $4$-vertex induced path depicted bold in the figure corresponds to a $4$-edge path in $K_{2,3}$ that does not close to a cycle.

By \Cref{thm:avoidable-edges}, the middle edge of any induced $P_4$ in an edge-transitive graph is avoidable. This implies the following consequence for edge-transitive graphs.

\begin{corollary}\label{cor:edge-transitive-induced-P_4}
Every induced $P_4$ in an edge-transitive graph closes to an induced cycle.
\end{corollary}


As shown by the example in Figure~\ref{fig:ET}, a statement analogous to that of \Cref{cor:edge-transitive-induced-P_4} fails for longer paths. This $13$-vertex example was found using a computer-assisted search based on a catalogue of vertex-transitive graphs due to~\citet{MR1088288} (which has been extended in the years since then) and performed using SageMath~\cite{sage}. Once the graph was found, a drawing of it given in the figure was identified with the help of~House of Graphs~\cite{MR2973372}.

\begin{figure}[!ht]
    \centering
    \begin{tikzpicture}[vertex/.style={inner sep=2pt,draw,circle},auto]

        \foreach \a in {0,...,12}
        {
            \node[vertex] (\a) at ({\a*(360/13)}:2.5){};
        }

        \draw[] (12)-- (0);

        \foreach \a in {0,...,11}
        {
            \pgfmathtruncatemacro{\b}{\a +1}
            \draw[] (\a) -- (\b);
        }

        \foreach \a in {0,...,12}
        {
            \pgfmathtruncatemacro{\b}{\intcalcMod{\a+5}{13}}
            \draw[] (\a) -- (\b);
        }

        \draw[line width = 1.25] (0)--(12)--(4)--(9)--(10);

        \node[label=right:$ v_1 $] (v1) at (0) {};
        \node[label=right:$ v_2 $] (v2) at (12) {};
        \node[label=above:$ v_3 $] (v3) at (4) {};
        \node[label=below:$ v_4 $] (v4) at (9) {};
        \node[label=below:$ v_5 $] (v5) at (10) {};
    \end{tikzpicture}
    \caption{An edge-transitive graph with an induced five-vertex path $(v_1,v_2,v_3,v_4,v_5)$ that does not close to an induced cycle. In other words, the induced three vertex path $ (v_2,v_3,v_4) $ is not avoidable. However, there also exists an avoidable $ P_3 $, in accordance with \Cref{conj:paths}.}
\label{fig:ET}
\end{figure}
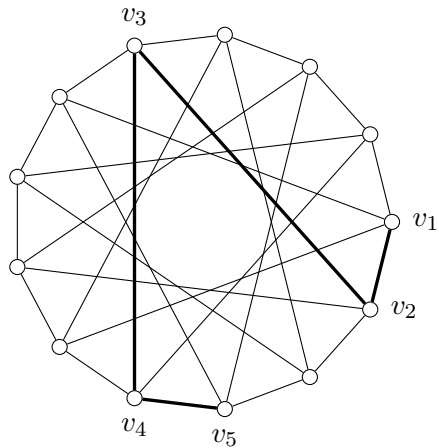

%% file: open-problems.tex

We have introduced the notion of avoidability in graphs, a concept that has been implicitly used in a variety of contexts in algorithmic graph theory.
We discussed both structural and algorithmic aspects of avoidable vertices, including a characterization of avoidable vertices as simplicial vertices in some minimal triangulation of the graph, a new proof of the existence of avoidable vertices, and the fact that one or two avoidable vertices in a graph can be found in linear time using a simple application of Lexicographic Breadth-First Search. This approach was then used to construct a polynomial-time algorithm for the maximum weight clique problem in the class of $1$-perfectly orientable graphs and a superclass of these, the graphs admitting an orientation in which every hole is oriented cyclically. We suggested a generalization of the concept of avoidability from vertices to nontrivial induced paths and proposed a conjecture about their existence (\Cref{conj:paths}). In this respect we showed the validity of the conjectures for edges, that is, two-vertex paths. Many interesting questions remain open.

\begin{sloppypar}
The main one is \Cref{conj:paths}. Theorems~\ref{thm:avoidable} and~\ref{thm:avoidable-edges} imply that it is true for $k \in \{1,2\}$. In turn, this fact and \Cref{thm:higher-chordality} imply that the conjecture is true for the class of $\{C_6,C_7,\ldots\}$-free graphs, which includes several well studied graph classes such as weakly chordal graphs, cocomparability graphs, and AT-free graphs.
\end{sloppypar}

While we have given a linear-time algorithm to compute two distinct avoidable vertices in any nontrivial graph (\Cref{thm:computation}), it would also be of interest to devise an algorithm to compute \emph{all} avoidable vertices that is more efficient than the na\"ive approach.


\begin{sloppypar}
Having introduced the class of hole-cyclically orientable graphs as a generalization of $1$-perfectly orientable graphs, we can ask for structural properties of these graphs. In particular, it is not known whether they can be recognized in polynomial time. The complexity of the maximum independent set and $k$-coloring problems (for fixed $k\ge 3$) is also open both for $1$-perfectly orientable and for hole-cyclically orientable graphs.
\end{sloppypar}